\documentclass[11pt,reqno]{amsart}
\usepackage{amssymb,mathrsfs,graphicx}
\usepackage{ifthen}
\usepackage{colortbl}
\usepackage{enumerate}  
\definecolor{black}{rgb}{0.0, 0.0, 0.0}
\definecolor{red}{rgb}{1.0, 0.5, 0.5}

\provideboolean{shownotes} 
\setboolean{shownotes}{true} 
\newcommand{\margnote}[1]{
\ifthenelse{\boolean{shownotes}}%
{\marginpar{\raggedright\tiny\texttt{#1}}}%
{}%
}
\newcommand{\hole}[1]{
\ifthenelse{\boolean{shownotes}}%
{\begin{center} \fbox{ \rule {.25cm}{0cm} \rule[-.1cm]{0cm}{.4cm}
\parbox{.85\textwidth}{\begin{center} \texttt{#1}\end{center}} \rule
{.25cm}{0cm}}\end{center}} {} }

\title[Emergent dynamics of the Cucker-Smale flocking model and its variants]{Emergent dynamics of the Cucker-Smale flocking model and its variants}

\author[Choi]{Young-Pil Choi}
\address[Young-Pil Choi]{\newline Fakult\"at f\"ur Mathematik
    \newline  Technische Universit\"at M\"unchen, Boltzmannstra{\ss}e 3, 85748, Garching bei M\"unchen, Germany}
\email{ychoi@ma.tum.de}

\author[Ha]{Seung-Yeal Ha}
\address[Seung-Yeal Ha]{\newline Department of Mathematical Sciences and Research Institute of Mathematics
    \newline Seoul National University, Seoul 151-747, Republic of Korea}
\email{syha@snu.ac.kr}

\author[Li]{Zhuchun Li}
\address[Zhuchun Li]{\newline Department of Mathematics
    \newline Harbin Institute of Technology, Harbin 150001, P.R.China}
\email{lizhuchun@hit.edu.cn}

\numberwithin{equation}{section}

\newtheorem{theorem}{Theorem}[section]
\newtheorem{lemma}{Lemma}[section]

\newtheorem{proposition}{Proposition}[section]
\newtheorem{remark}{Remark}[section]
\newtheorem{definition}{Definition}[section]

\newcommand{\R}{\mathbb R}

\newcommand{\T}{\mathbb T}

\newcommand{\bq}{\begin{equation}}
\newcommand{\eq}{\end{equation}}

\newcommand{\lt}{\left}
\newcommand{\rt}{\right}

\newcommand{\pa}{\partial}

\newcommand{\ml}{\mathcal{L}}

\newcommand{\om}{\Omega}

\begin{document}

\allowdisplaybreaks

\begin{abstract}In this chapter, we present the Cucker-Smale type flocking models, and discuss their mathematical structures and flocking theorems in terms of coupling strength, interaction topologies and initial data. In 2007, two mathematicians Felipe Cucker and Steve Smale introduced a second-order particle model which resembles Newton's equations in $N$-body system, and present how their simple model can exhibit emergent flocking behavior under sufficient conditions expressed only in terms of parameters and initial data. After Cucker-Smale's seminal works in \cite{CS1, CS2}, their model has received lots of attention from applied math and control engineering  communities. We discuss the state-of-art for the flocking theorems to Cucker-Smale type flocking models.
\end{abstract}

\maketitle 

\tableofcontents

\section{Introduction} \label{Sec1}
The jargon ``{\it flocking}" represents collective phenomena in which self-propelled particles (or agents) is organized into an ordered motion from a disordered state using only limited environmental information and simples rules \cite{TT}. Such an organized motion is ubiquitous in our nature, e.g., aggregation of bacteria, flocking of birds, swarming of fish, herding of sheep etc. \cite{BC, VCBCS}, and they have been extensively studied recently because of their possible applications to sensor networks, controls of robots and unmanned aerial vehicles \cite{LPLS, PLSGP, PGE}, and opinion formation of social networks. After the pioneering work \cite{R, VCBCS} of Reynolds and Vicsek et al, many agent-based models have been proposed in literature and studied extensively both analytically and numerically. Among them, we are interested in the model introduced by Cucker and Smale \cite{CS1, CS2}.  This model resembles Newton type $N$-body system for an interacting particle system. In the sequel, we introduce  Cucker-Smale (C-S) type models  from microscopic to macroscopic sales and under various network topologies. We also summarize state-of-art flocking theorems for the C-S type models and explain how  these models can achieve asymptotic flocking under what conditions and main ideas behind them. For other survey on the related topics, we refer to \cite{CFTV, MT0}. \newline

The rest of this chapter is organized as follows. In Section 2, we present hierarchical models for the description of C-S flocking ensemble starting from the particle to kinetic and fluid desciptions. In Section 3, we introduce three continuous-time C-S type models including the original flocking model \cite{CS2} and discuss the flocking problem for these models.  In Section 4, we present a discrete-time C-S model with leadership structures such as hierarchical and rooted leaders, alternating leadership. In Section 5, we present a mesoscopic description, namely kinetic picture for the C-S flocking. We also discuss flocking particle-fluid interactions via the coupled kinetic-fluid model. In Section 6, we present a C-S hydrodynamic flocking model and its flocking estimate, and then we study its coupling with compressible Navier-Stokes equations through the drag force. \newline

\noindent {\bf Notation}: Throughout the chapter we use a superscript to denote the component of a vector; for example $x :=(x^1, \dots, x^d) \in {\mathbb R}^d$. Subscripts are used  to represent the ordering of particles.  For vectors $x, v \in {\mathbb R}^d$, its $\ell_2$-norm and the inner product are defined as follows:
\[
|x|:=\left(\sum_{i=1}^d (x^i)^2\right)^{\frac{1}{2}}, \qquad
\langle x, v \rangle:=\sum_{i=1}^d x^i v^i,
\]
where $x^i$ and $v^i$ are the $i$-th components of $x$ and $v$, respectively.

\section{Preliminaries} \label{Sec2}
\setcounter{equation}{0}
In this section, we briefly discuss hierarchical C-S models from microscopic scale to mesoscopic and macroscopic scales following the presentation in \cite{HT}. \newline

In \cite{CS1,CS2}, Cucker and Smale introduced a Newton type microscopic model for an interacting many-body system exhibiting a flocking phenomenon, and provide sufficient conditions for an asymptotic flocking (see Definition \ref{D1}) in terms of initial configuration and interaction topologies. We next describe the C-S model. Let $x_i$ and $v_i$ to be the position and velocity of the $i$-th C-S particle, respectively. Then, the C-S model with metric dependent communication weight $\psi$ is given by the following ODE system:
\begin{align}
\begin{aligned} \label{CS}
\frac{dx_i}{dt} &= v_i, \quad t > 0,~~i=1, \dots, N, \\
\frac{dv_i}{dt} &= \frac{K}{N} \sum_{j=1}^{N} \psi(|x_j - x_i|) (v_j - v_i),
\end{aligned}
\end{align}
where $K$ is a nonnegative coupling strength and $\psi$ is a communication weight measuring the degree of communications(interactions) between particles.  For a large C-S system \eqref{CS} with $N \gg 1$, it is not reasonable to integrate the particle model for computational purpose, because it is too expensive to integrate \eqref{CS} numerically even if it is possible. Thus, it is natural to introduce a kinetic model as an approximation for \eqref{CS}. For this, we introduce a kinetic density (one-particle distribution function) $f=f(x,\xi,t)$ at phase space position $(x,\xi)$, at time $t$. Then the spatial-temporal evolution of $f$ is governed by the following Vlasov-McKean equation:
\begin{align}
\begin{aligned} \label{KCS}
& \partial_t f + \xi \cdot \nabla_x f + \nabla_\xi \cdot (F_a(f) f) = 0, \quad (x, \xi) \in \R^d \times \R^d,~t > 0, \\
& F_a(f)(x,\xi,t) = -K \int_{\R^{2d}} \psi(|x-y|) (\xi- \xi_*) f(y,\xi_*) d\xi_* dy.
\end{aligned}
\end{align}
The equation \eqref{KCS} admits a global smooth solution, as long as initial datum is compactly supported in $x$ and $v$ and sufficiently regular (see \cite{HT}). In kinetic theory of gases, it is well known that the velocity moments of $f$ yield the macroscopic observables. For example, for a given $(x,t) \in \R^d \times \R_+$, we set
\begin{align}
\begin{aligned} \label{moments}
\rho &:= \int_{\R^d} f d\xi: \quad \mbox{local mass density}, \\
\rho u &:= \int_{\R^d} \xi f d\xi: \quad \mbox{local momentum density}, \\
\rho E  &:= \rho e + \frac{1}{2} \rho |u|^2: \quad \mbox{local energy density},
\end{aligned}
\end{align}
where $\displaystyle \rho e := \frac{1}{2} \int_{\R^d} |\xi - u(x)|^2 f d\xi$ is the internal energy. Then, macroscopic observables \eqref{moments} satisfy the following hydrodynamic equations:
\begin{align}
\begin{aligned} \label{macro}
& \partial_t \rho + \nabla_x \cdot (\rho u) = 0, \quad x \in \R^d,~~t > 0, \\
& \partial_t (\rho u) + \nabla_x \cdot (\rho u \otimes u + P) = S^{(1)}, \\
& \partial_t (\rho E) + \nabla_x \cdot ( \rho E u + P u + q) = S^{(2)},
\end{aligned}
\end{align}
where $P = (p_{ij})$ and $q = (q_1, \dots, q_d)$ are stress tensor and heat flow, respectively.
\begin{equation} \label{stress}
 p_{ij} := \int_{\R^d} (\xi_i - u_i) (\xi_j - u_j) f d\xi, \quad q_i := \int_{\R^d} (\xi_i - u_i) |\xi-u|^2 f d\xi,
\end{equation}
and the source terms are given by the following relations:
\begin{align}
\begin{aligned} \label{source}
S^{(1)} &:= -K \int_{\R^d} \psi(|x-y|) (u(x) - u(y)) \rho(x) \rho(y) dy, \cr
S^{(2)} &:= -K \int_{\R^d} \psi(|x-y|) (E(x) + E(y) - u(x) \cdot u(y)) \rho(x) \rho(y) dy.
\end{aligned}
\end{align}
Of course, the moment system \eqref{macro} is not closed as it is, because we need to know the third velocity moment of $f$ to calculate the heat flux $q$ in \eqref{stress}. So far, suitable closure conditions
for \eqref{macro}(e.g., the local Maxwellian  for the Botlzmann equation) are not known. In a quasi-flocking regime, we may employ the mono-kinetic ansatz for $f$:
\begin{equation} \label{mono}
 f(x,\xi,t) = \rho(x,t) \delta(\xi - u(x,t)), \quad x, \xi \in \R^d, ~t > 0.
\end{equation}
Then, under this mono-kinetic assumption \eqref{mono}, the stress tensor $P = (P_{ij})$ and heat flux $q$ become zero:
\[
 p_{ij} = 0, \qquad q_i = 0, \quad 1 \leq i, j \leq d.
\]
Thus, in the quasi-flocking regime, the system \eqref{macro}-\eqref{source} is reduced to the pressureless Euler system with a flocking dissipation:
\begin{align}
\begin{aligned} \label{macro-1}
\partial_t \rho + \nabla_x \cdot (\rho u) &= 0, \quad x \in \R^d,~~t > 0, \\
\partial_t (\rho u) + \nabla_x \cdot (\rho u \otimes u) &=  -K \rho \int_{\R^d} \psi(|x-y|)  \rho(y) (u(x) - u(y)) dy, \\
\end{aligned}
\end{align}
Note that the energy equation in \eqref{macro} can be derivable from the equations for $\rho$ and $\rho u$, and the condition \eqref{mono} will be
valid only for the collisionless regime. However, when particles with different microscopic velocities collide, the mono-kinetic ansatz \eqref{mono}
will break down. Therefore, our system \eqref{macro-1} should be regarded as a quasi equilibrium model for the hydrodynamic description of the C-S ensemble.

\begin{remark}\label{rem_max}
If we consider other strong interaction forces such as local alignment and noise, then the density function $f$ is close to a thermodynamical equilibrium $f \sim C_0 \rho e^{-|u - \xi|^2/2}$, and in this case, the dynamics can be well approximated by a compressible isothermal Euler equations with the velocity-alignment force. This rigorous derivation is obtained in \cite{KMT} by employing a relative entropy argument.
\end{remark}

%
%
%
\section{Continuous-time Cucker-Smale type models} \label{continuous}
In this section, we discuss continuous-time C-S models and their flocking estimates.  As discussed in previous section, after Cucker-Smale's seminal works in \cite{CS1, CS2}, several variants of the C-S model have been introduced for better modelings including local and nonsymmetric interactions, collision avoidance and formation control, etc in \cite{MT, PKH}. In the following, we explain how flocking estimates for particle models can be obtained. We first consider a Cauchy problem for the C-S model:
\begin{align}
\begin{aligned} \label{CS}
&{\dot x}_i= v_i,  \quad t > 0, \quad i =1, \dots, N, \\
&{\dot v}_i = \frac{K}{N} \sum_{j=1}^{N} \psi(|x_j - x_i|) (v_j - v_i),
\end{aligned}
\end{align}
subject to initial data
\begin{equation} \label{Ini-1}
(x_i, v_i)(0) = (x_{i0}, v_{i0}),
\end{equation}
where the communication weight  function $\psi: {\mathbb R}_+ \to {\mathbb R}$ is assumed to be Lipschitz continuous, nonnegative and non-increasing:
\begin{equation} \label{comm-weight}
\psi \in \mbox{Lip}(\R_+; \R), \quad \psi \geq 0, \quad (\psi(r_2) - \psi(r_1))(r_2 - r_1) \geq 0, \quad r_1, r_2 \geq 0.
\end{equation}
Before we present flocking estimates for \eqref{CS} - \eqref{Ini-1}, we recall the definition of (mono-cluster) flocking of a many-body system as follows.
\begin{definition} \label{D1}
\emph{\cite{CS1, HT}}
Let ${\mathcal G} := \{(x_i, v_i)\}_{i=1}^N$ be an $N$-body interacting system. Then ${\mathcal G}$ exhibits a asymptotic flocking if and only if the following two relations hold.
\begin{enumerate}
\item
(Velocity alignment): The relative velocities approach to zero asymptotically.
\[ \lim_{t \to \infty} |v_i(t) - v_j(t)| = 0, \quad 1 \leq i, j \leq N. \]
\item
(Spatial coherence): The relative positions are uniformly bounded:
\[ \sup_{0 \leq t < \infty} |x_i(t) - x_j(t)| < \infty, \quad 1 \leq i, j \leq N. \]
\end{enumerate}
\end{definition}
To have some feeling for the large-time dynamics of \eqref{CS}, we consider the simplest system made of two C-S particles on the real line $\R$:
\begin{align}
\begin{aligned} \label{two}
& {\dot x}_1 = v_1, \quad {\dot x}_2 = v_2,  \quad t > 0,~~x_i, v_i \in {\mathbb R}, \\
& {\dot v}_1= \frac{K}{2} \psi(|x_2 - x_1|) (v_2 - v_1), \quad {\dot v}_2 = \frac{K}{2} \psi(|x_1 - x_2|) (v_1 - v_2), \\
& (x_i, v_i)(0)= (x_{i0}, v_{i0}).
\end{aligned}
\end{align}
To reduce the number of equations in \eqref{two}, we introduce the spatial and velocity differences:
\[ x := x_1 - x_2, \quad v := v_1 - v_2. \]
Then, without loss of generality, we may assume
\begin{equation} \label{ini-con}
x_0 > 0, \quad v_0 > 0.
\end{equation}
Note that the differences of $x$ and $v$ satisfy
\[
{\dot x} = v, \quad {\dot v} = -K \psi(|x|) v,
\]
or equivalently,
\[  dv = - K \psi(|x|) dx. \]
We integrate the above relation to obtain
\begin{equation} \label{two-rel}
 v(t) =  v_0 -K \int_{x_0}^{x(t)} \psi(|y|) dy.
\end{equation}
Depending on the relations between the coupling strength $K$ and initial data, we might not have flocking in the sense of Definition \ref{D1}. This negative result can be seen from the following proposition.

\begin{proposition}  \label{P1}
\emph{\cite{CHHJK1}} Suppose that the communication weight $\psi$ takes the following form:
\begin{equation} \label{AAA}
 \psi(|x - y|) = \frac{1}{(1 + |x-y|)^{\beta}}, \quad \beta \geq 0,
\end{equation} 
and let $(x,v)$ be the solution to the system \eqref{two}-\eqref{ini-con} with initial data $(x_0, v_0)$. Then the following assertions hold:
\begin{enumerate}
\item
If $(x_0, v_0)$ satisfies
\begin{equation}\label{New-1}
\displaystyle  v_0 = K \int_{x_0}^{\infty} \psi(|y|) dy,
\end{equation}
then the positions of the two particles diverge with the {\it same} asymptotic velocities.
\item
If $(x_0, v_0)$ satisfies
\begin{equation} \label{New-2}
 v_0 > K \int_{x_0}^{\infty} \psi(|y|) dy,
\end{equation}
then the positions of the two particles diverge with {\it different} asymptotic velocities.
\end{enumerate}
\end{proposition}
\begin{proof} (i) Suppose $(x_0, v_0)$ satisfies
\[ v_0 = K \int_{x_0}^{\infty} \psi(|y|) dy. \]
Using \eqref{comm-weight}, \eqref{two-rel}, and \eqref{New-1}, we obtain
\begin{equation} \label{AAA-1}
 v(t)  =  v_0 -K \int_{x_0}^{x(t)} \psi(|y|) dy = K\int_{x(t)}^{\infty} \psi(|y|) dy > 0.
\end{equation}
On the other hand, since $\frac{dx}{dt} = v > 0$ and $x(0) = x_0 > 0$, we have
\begin{equation} \label{AAA-2}
 x(t) > 0, \quad \mbox{i.e.,} \quad  \psi(|x|) = \psi(x). 
\end{equation} 
We now use \eqref{AAA}, \eqref{AAA-1} and \eqref{AAA-2} to find a first-order equation for $x$:
\begin{equation} \label{New-3}
 \frac{dx}{dt}= K\int_{x(t)}^{\infty} \psi(y) dy =  \frac{K}{\beta-1} \frac{1}{\left(1+x(t)\right)^{\beta-1}}.
\end{equation}
Directly integration \eqref{New-3} yields
\[
  \displaystyle x(t)= \left(\frac{\beta Kt}{\beta-1}+(1+x_0)^\beta\right)^{1/\beta}-1, \qquad  v(t)=\frac{K}{\beta-1}\left(\frac{\beta Kt}{\beta-1}+(1+x_0)^\beta\right)^{1/\beta-1}.
 \]
The above explicit formula implies
\[ \lim_{t \to \infty} x(t) = \infty, \qquad \lim_{t \to \infty} v(t) = 0. \]
Note that the velocity difference of $v$ goes to zero at the rate of $t^{-(1-1/\beta)}$.  \newline

\noindent (ii) Suppose $(x_0, v_0)$ satisfies \eqref{New-2}. It follows from \eqref{two-rel} that
\begin{align}
\begin{aligned} \label{New-4}
v(t) &=  v_0 -K \int_{x_0}^{x(t)} \psi(|y|) dy \\
      &= v_0 - K \int_{x_0}^{\infty} \psi(|y|)dy + K \int_{x(t)}^{\infty} \psi(|y|) dy.
\end{aligned}
\end{align}
Note that \eqref{New-4} implies
\[ v(t) \geq v_0 - K \int_{x_0}^{\infty} \psi(|y|)dy > 0, \quad t \geq 0.\]
Thus, the asymptotic velocities are not equal. On the other hand, if we set $$ v_{\infty} := v_0 - K
\int_{x_0}^{\infty} \psi(|y|) dy,$$ then \eqref{New-4} implies
\[
\frac{dx}{dt} = v_{\infty} +
\frac{K}{\beta-1}\left(1+x(t)\right)^{1-\beta}.
\]
Clearly, $x(t)$ increases faster than $ v_{\infty}t$ by the comparison theorem. This completes the proof.
\end{proof}
\begin{remark} \label{R2}
1. It follows from Proposition \ref{P1} that even for a simple two-body system, the flocking theorem is not always true and it depends on the interplay between the coupling strength and initial data. In the following three subsections, we briefly discuss the flocking estimates for the continuous-time C-S model. In the following three subsections, we present three variants of the C-S model with metric dependent communications. \newline

\noindent 2. The coupling strength function \eqref{AAA} appears also in different forms in literature:
\[\psi(|x-y|):~\frac{1}{(1+|x_i(t)-x_j(t)|^2)^\frac{\beta}{2}} \quad \text{or}\quad \frac{1}{(1+|x_i(t)-x_j(t)|)^\beta}.\]
They are in fact equivalent to each other. Thus we may use  $\beta \leftrightarrow \frac{\beta}{2}$ in the models and results interchangeably.
 \end{remark}
\subsection{General symmetric weights}\label{sec_gsw} In this subsection, we briefly review sufficient conditions for the emergence of asymptotic flocking  for the C-S model in \eqref{CS}-\eqref{comm-weight}. Flocking estimate was first studied by Cucker and Smale \cite{CS1}. They provided a sufficient condition on the formation of flocking for an algebraically decaying communication weight $\psi(r) = (1 + r^2)^{-\beta/2}$ with $\beta \geq 0$. For the short-ranged communication weight, they showed that asymptotic flocking is possible for initial configurations close to the flocking state using the self-bounding argument. Later, Cucker and Smale's results were further generalized to general nonincreasing communication weights \eqref{comm-weight} using a simpler energy method and Lyapunov functional approach, which were  based on the $\ell^2$-norm and mixed $\ell^{\infty}-\ell^2$ norms in ~\cite{ACHL, HL, HT}.  For a given configuration $(x, v) \in R^{2dN}$ with a zero sum condition:
\[ \sum_{i=1}^{N} x_i(t) = 0, \qquad  \sum_{i=1}^{N} v_i(t) = 0, \quad t \geq 0, \]
we set
\[
|x|_{\infty} := \max_{1 \leq i \leq N} |x_i|, \quad |v|_{\infty} := \max_{1 \leq i \leq N} |v_i|. \]
Then, the norms $|x|_{\infty}$ and $|v|_{\infty}$ are Lipschitz continuous, and satisfy a system of dissipative differential inequalities:
\begin{equation} \label{SDDI}
\lt| \frac{d}{dt} |x|_{\infty} \rt| \leq |v|_{\infty}, \quad \frac{d}{dt} |v|_{\infty} \leq -K \psi(2|x|_{\infty})
|v|_{\infty} \quad \mbox{a.e.}~~t \in (0, \infty).
\end{equation}
Note that once we have a uniform bound for $|x|_{\infty}$, the second relation in \eqref{SDDI} yields the exponential decay of $|v|_{\infty}$. Thus, we introduce Lyapunov-type functionals ${\mathcal L}_{\pm}(t)
\equiv {\mathcal L}_{\pm}(x(t), v(t))$:
\begin{equation} \label{Lya}
  {\mathcal L}_{\pm}(t) :=  |v(t)|_{\infty} \pm \frac{K}{2} \int_{0}^{2|x(t)|_{\infty}} \psi(s) ds, \quad t \geq 0.
\end{equation}
Then, it is easy to see the non-increasing property of ${\mathcal L}_{\pm}$ using  \eqref{SDDI}:
\[
{\mathcal L}_{\pm}(t) \leq {\mathcal L}_{\pm}(0), \quad t \geq 0,
\]
which leads to the stability estimate of ${\mathcal L}_{\pm}(t)$:
\[
|v(t)|_{\infty} + \frac{K}{2}  \lt| \int_{2|x_0|_{\infty}}^{2|x(t)|_{\infty}} \psi(s)  ds \rt| \leq |v_0|_{\infty}, \quad t \geq 0.
\]
The following theorem is most relevant result on the flocking estimate.
\begin{theorem} \label{T1}
\emph{\cite{ACHL, CS1, HL, HT}}.
Let $(x, v)$ be a solution to \eqref{CS}-\eqref{comm-weight} with initial data $(x_0, v_0)$ satisfying the following condition:
\begin{equation} \label{Cond}
 |x_0|_{\infty} > 0, \quad  |v_0|_\infty <  \frac{K}{2} \int_{|x_0|_\infty}^{\infty} \psi(2r) dr.
 \end{equation}
Then there exists a positive number $x_M$ such that
\[  \sup_{t \geq 0} |x(t)| \leq x_M, \quad  |v(t)| \leq |v_0| e^{-\psi(2x_M) t}, \quad t \geq 0. \]
\end{theorem}
\begin{remark}
1. The result of Theorem \ref{T1} can be restated as follows. For a given initial data $(x_0, v_0)$, there exists a coupling strength $K^*(x_0, v_0) =: 2|v_0|_\infty/\int_{|x_0|_\infty}^{\infty} \psi(2r) \,dr$ such that if $K > K^*(x_0, v_0)$, then we have an exponential flocking. Thus, natural question is to see the large-time dynamics in the regime $K < K^*(x_0, v_0)$ which Theorem \ref{T1} cannot be applied for. In this small coupling strength regime, there might be local flocking(or multi-cluster flocking). Recently this issue has been addressed in a series of papers \cite{CHHJK1, CHHJK2, HKZZ, HKZ}. \newline

\noindent 2. In \cite{ACHL}, C-S model with a singular communication weight is considered. Under certain condition of the initial configurations, the collision avoidance between agents is provided. Later, these conditions are refined in \cite{CCP}.
\end{remark}

\subsection{Non-symmetric interactions} After Cucker-Smale's seminal work \cite{CS2}, one interesting extension of the C-S model has been proposed by Motsch and Tadmor in \cite{MT}. They replaced the symmetric interaction potential $\psi_{ij}$ in the C-S model by a non-symmetric one:
\[ \psi_{ij} \quad \longrightarrow \quad \frac{\psi_{ij}}{\sum_{k=1}^{N} \psi_{ik}}. \]
Thus, the general C-S model proposed by Motsch and Tadmor reads as follows.
\begin{align}
\begin{aligned} \label{CS-MT}
\displaystyle \frac{dx_i}{dt} &= v_i, \quad t > 0,~~i=1, \dots, N, \\
\displaystyle \frac{dv_i}{dt} &= \frac{K}{\sum_{k=1}^{N} \psi(|x_k - x_i|)} \sum_{j=1}^{N} \psi(|x_j - x_i|) (v_j - v_i).
\end{aligned}
\end{align}
This model does not only take into account the distance between agents, but instead, the influence between
agents is scaled in term of their relative distance. Hence, it does not involve any explicit dependence on the number of agents. However, this extension of communication weight destroys the symmetry property of the original C-S model. The symmetry property of the communication weights is essential in energy estimate for the C-S model. Fortunately, the Lypaunov type functional approach introduced in the previous subsection works for this non-symmetric situation. To state their result, we introduce diameters for $x$ and $v$:
\[  D(x) := \max_{1 \leq i, \leq N} |x_i - x_j |, \quad D(v) := \max_{1 \leq i, \leq N} |v_i - v_j |.  \]
Then, by the detailed calculations, they showed that these diameters satisfy a system of dissipative differential inequalities:
\[
\lt| \frac{d D(x)}{dt} \rt| \leq D(v), \quad \frac{d D(v)}{dt} \leq -K \psi^2(D(x)) D(v), \quad \mbox{a.e.,}~t \in (0, \infty).   \]
Then, by using the idea of Lyapunov functional approach \eqref{Lya} depicted in the previous subsection, they obtain the following flocking estimate.
\begin{theorem}
\emph{\cite{MT}}
Suppose that $\psi$ is positive and initial data satisfy
\[ D(v_0) < \int_{D(x_0)}^{\infty} \psi^2(r) dr. \]
Then, the model \eqref{CS-MT} exhibits an asymptotic flocking:
\[ \sup_{0 \leq t < \infty} D(x(t)) < \infty \quad \mbox{and} \quad \lim_{t \to \infty} D(v(t)) = 0. \]
In particular, if $\psi$ has a fat tail such that
\[ \int_c^{\infty} \psi^2(r) dr = \infty, \quad \mbox{for any positive $c$}, \]
then asymptotic flocking occurs for any initial data.
\end{theorem}
\subsection{Bonding force} For the realistic applications to robotic multi-agent systems, we need to consider the formation control and collision avoidance. For this, we extend the C-S model by introducing additional interaction terms between agents in order to incorporate collision
avoidance between agents, and at the same time achieve tighter spatial configurations. We make use of not only
position but also velocity information of the agents in order to derive
the additional interaction between agents. This results in a control term
which drives agents together or away in such a manner that the distance between agents converge to a non-zero constant value. The C-S model with aforementioned formation control and collision avoidance terms reads as follows.
\begin{align}
\begin{aligned} \label{CS-bond}
&{\dot x}_i= v_i,  \quad t > 0, \quad i =1, \dots, N, \\
&{\dot v}_i = \frac{K_0}{N} \sum_{j=1}^{N} \psi(|x_j - x_i|) (v_j - v_i) + \frac{K_1}{N} \sum_{j=1}^{N} \frac{\langle v_j - v_i, x_j - x_i \rangle}{|x_j  - x_i|} (x_j - x_i) \\
& \hspace{0.4cm} + \frac{K_2}{N} \sum_{j=1}^{N} (|x_j - x_i| - 2R) (x_j - x_i),
\end{aligned}
\end{align}
where $K_0, K_1$ and $K_2$ are nonnegative coupling constants. Due to the translation invariance of \eqref{CS-bond}, without loss of generality, we assume that
\begin{equation} \label{zerosum}
 \sum_{i=1}^{N} x_i(t) = 0, \quad  \sum_{i=1}^{N} v_i(t) = 0, \quad t \geq 0.
\end{equation}
 We define energy functionals:
\[  {\mathcal E} := {\mathcal E}_k + {\mathcal E}_p, \quad {\mathcal E}_k := \frac{1}{2} \sum_{i=1}^{N} |v_i|^2, \quad {\mathcal E}_p := \frac{K_2}{4N} \sum_{1 \leq i, j \leq N}\lt(|x_j - x_i | - 2R \rt)^2, \]
where ${\mathcal E}_k$ and ${\mathcal E}_p$ represent kinetic and potential energies, respectively. Then, it follows from the energy estimates that the total energy ${\mathcal E}$ satisfies dissipation estimate.
\begin{proposition}
\emph{\cite{PKH}}
For some $T \in (0, \infty]$, let $(x, v)$ be a solution to \eqref{CS-bond} - \eqref{zerosum} in the time-interval $[0, T)$. Then, the energy functional ${\mathcal E}$ is non-increasing in time $t$:
\[ {\mathcal E}(t) + \int_0^t {\mathcal P}(\tau) d\tau = {\mathcal E}(0), \quad t \geq 0, \]
where energy production functional $P$ is given by the relation:
\[ {\mathcal P} := \frac{K_0}{2N} \sum_{1\leq i,j \leq N} \psi(|x_j - x_i|) |v_j - v_i|^2 + \frac{K_1}{2N} \sum_{1 \leq i, j \leq N} \lt( \frac{d}{dt} |x_j - x_i |^2 \rt).     \]
\end{proposition}
This yields the flocking estimate for \eqref{CS-bond}.
\begin{theorem}
\emph{\cite{ACHL, PKH}}
Suppose that the communication weight $\psi$ and initial data satisfy the following conditions.
\begin{eqnarray*}
&& \psi(r) \geq 0, \quad \exists~r_0 \in (0, \infty] \quad \mbox{such that }~\psi(r) > 0, \quad \mbox{for}~r \leq r_0, \\[1mm]
&& \psi_m := \min \lt\{ \psi(r)~:~ 0 \leq r \leq 2R + \sqrt{\frac{2N {\mathcal E}(0)}{K_2}} \rt\} > 0, \quad {\mathcal E}(0) < K_2 R^2 N.
\end{eqnarray*}
Let $(x,v)$ be a global solution to \eqref{CS-bond} - \eqref{zerosum}. Then, the following assertions hold.
\begin{enumerate}
\item
Asymptotic flocking occurs.
\[ \sup_{0 \leq t < \infty} |x_i(t) - x_j(t)| < 2R + \sqrt{\frac{2N {\mathcal E}(0)}{K_2}}, \quad \lim_{t \to \infty} |v_i(t) - v_j(t)| =0, \quad 1 \leq i, j \leq N.
\]
\item
The collision avoidance is guaranteed.
\[ \inf_{0 \leq t < \infty} |x_i(t) - x_j(t)| > 0, \quad 1 \leq i, j \leq N. \]
\end{enumerate}
\end{theorem}

\section{Discrete-time Cucker-Smale models with leadership}\label{discrete}


In this section, we present the discrete-time C-S model. 
Let $(x_i, v_i)$ denote the position and velocity of the $i$-th particle, 
then the discrete-time C-S model, with time step $h>0$, is governed by \begin{align}
\begin{aligned} \label{c-s-1}
x_i(t+1)&=x_i(t)+hv_i(t), \qquad  i=1,2,\dots,N,\\
 v_i(t+1)&=v_i(t)+h\sum\limits_{j=1}^N \phi(|x_i - x_j|)\left(v_j(t)-v_i(t)\right),\\
   \phi(|x_i - x_j|)&=\frac{1}{(1+|x_i(t)-x_j(t)|^2)^\beta}, \quad\,\, \beta \geq 0.
\end{aligned}
\end{align}


    The   choice of weight function   is a crucial ingredient which
  makes the C-S model attractive: the convergence results depend on conditions on the
initial state only. In contrast,   the convergence results in \cite{JLM} for the linearized Vicsek model rely on
some assumptions on the infinite time-sequence of states.   In the original C-S model, the interactions are  bidirectional, thus, symmetric.    With  bidirectional couplings, they used the Fiedler number of the symmetric
Laplacian matrix to develop some estimates on the iterates of the fluctuations of position and velocity so that the  self-bounding lemma \cite[Lemma 2]{CS1} can be applied.   The pioneering work \cite{CS1} of Cucker and Smale gave the following  flocking theorem.
\begin{theorem}\cite{CS1} Consider the  model \eqref{c-s-1} (which is under all-to-all and symmetric coupling). If $\beta<\frac{1}{2}$, the flocking occurs for any initial data; if $\beta\geq \frac{1}{2}$, the flocking occurs depending on the initial data.
\end{theorem}

An example in \cite{CS1} shows that when $\beta>\frac{1}{2}$, the unconditional flocking is not true. Thus, the exponent $1$ is regarded as  the critical exponent for the unconditional flocking. In this section we will briefly introduce some results on discrete-time C-S model with interactions under some leadership. Before we discuss the variant interaction topologies, we first briefly introduce some concepts in graph theory \cite{D}.
A digraph  ${\mathcal{G}} =({\mathcal{V}},{\mathcal{E}} )$ (without self-loops)   representing $(N+1)$ particles with interaction in   C-S model,  is defined by
\[ {\mathcal{V}} :=\{0,1, \dots,N\}, \qquad  {\mathcal{E}} \subseteq {\mathcal{V}} \times {\mathcal{V}}\setminus  \{(i,i):~i\in {\mathcal V}\}. \]
We say $(j,i)\in{\mathcal{E}} $ if and only if  $j$ influences $i$. In this case, we also write $j\in \mathcal L(i)$. The graph ${\mathcal{G}}$ can be regarded as the information flow chart of a network structure; that is,  we   write
\[  j \to i ~~ \iff ~~ (j,i)\in {\mathcal{E}}. \]
A directed path from $j$ to $i$ (of length $n+1$) comprises a sequence
of distinct arcs of the form $j\to k_1\to k_2\to\dots\to k_n\to i$.
The distance from $j$ to $i$ is the length of the shortest path from $j$ to $i$.

\subsection{Hierarchical leadership}
 The general form of  a discrete-time C-S model is given by
\begin{align}\label{c-s-2}
\begin{aligned}
x_i(t+1)&=x_i(t)+hv_i(t), \qquad  i=0,1,\dots,N,\\
 v_i(t+1)&=v_i(t)+h\sum\limits_{j=0}^N \phi_{ij}(x(t))\left(v_j(t)-v_i(t)\right),\\
   \displaystyle\phi_{ij}(x(t))&=\left\{\begin{array}{c}     0\,,   \quad\qquad\quad\qquad \,\,   \, \mathrm{if}~ j\notin \mathcal L(i),\\ \displaystyle \phi(|x_i - x_j|),    \quad \quad\,\,  \mathrm{if}~ j\in \mathcal L(i). \end{array}\right.
\end{aligned}
\end{align}
  Here,  $\mathcal L(i)\subset \{0,1,\dots,N\}$,  regarded as the {\em leader set} of agent $i$, is the set of   agents which influence $i$ directly.
Thus the interaction topology of the C-S model is registered in the configuration of $\mathcal L(i)$, or equivalently, the adjacency matrix $\Phi_x:=(\phi_{ij}(x))$.

  Shen  extended the C-S flocking to an asymmetric structure  in which the interactions are  unidirectional. More precisely, he considered a C-S model under hierarchical leadership,
which means that the agents
 can be partially ordered in such a way that lower-rank
agents are led and only led by some agents of higher ranks.  The formal definition of   hierarchical leadership is as follows.

\begin{definition}\label{defhl} \cite{S}
An $(N+1)$-flock $\{0,1,\dots,N\}$  is said to be under {hierarchical leadership} if the following two statements hold:
\begin{enumerate}[(a)]
\item $j\in \mathcal L(i)$ implies that $j<i$;
\item for any $i>0$, $\mathcal L(i)\neq \emptyset$.
\end{enumerate}
\end{definition}
  Definition \ref{defhl} means that   the adjacency matrix is triangular
under a proper ordering of the agents. 
 For the continuous-time model with hierarchy, Shen used the induction method to prove the unconditional flocking for $\beta<\frac{1}{2}$.     The triangularity of the adjacency matrix is
the key for  the induction method; actually the   idea lies in the fact that the dynamics of  agents $\{0,1,\dots,N\}$ does not change if a new agent $N+1$   ranking lowest is added.  Later, Cucker and Dong extended the induction method to the discrete-time hierarchical model   to improve the critical exponent, see \cite{CD1}. The main result is as follows.
\begin{theorem}\cite{CD1,S} Consider the  model \eqref{c-s-2}  with interaction topology as in Definition \ref{defhl}. If $\beta<\frac{1}{2}$, the flocking occurs for any initial data.
\end{theorem}

The induction method does not give any sufficient condition for the flocking behavior when   $\beta$ is larger than the critical exponent, i.e., $\beta>\frac{1}{2}$.

\subsection{Individual preference}
A   variant of the hierarchical C-S model is  the  flocking  with individual preference \cite{L}:
\begin{align}
\begin{aligned} \label{cshlinform}
x_i(t+1)&=x_i(t)+hv_i(t), \qquad  i=0,1,\dots,N,\\
 v_i(t+1)&=v_i(t)+h\sum\limits_{j=0}^N \phi_{ij}(x(t))\left(v_j(t)-v_i(t)\right)+h\delta_i(t)q_i(t),\\
   \displaystyle\phi_{ij}(x(t))&=\left\{\begin{array}{c}     0\,,   \qquad\qquad\qquad \,\,   \, \mathrm{if}~ j\notin \mathcal L(i),\\ \displaystyle H\phi(|x_i - x_j|), ~\,  \quad\,\,  \mathrm{if}~ j\in \mathcal L(i). \end{array}\right.
\end{aligned}
\end{align}
Here,  the parameter $H>0$ is incorporated as a  measure of the strength of leader-follower interactions, $q_i(t)\in \mathbb R^3$  describes the temporarily preferred acceleration of agent $i$,  $\delta_i(t)\in\mathbb R$ is a local measure of the consensus at time $t$ which determines the strength of preferred acceleration.  As a  special case, we may choose $q_i(t)\equiv \bar q_i$ for all time $t$; then it represents a constant preferred acceleration of agent $i$.  In \cite{L} a  typical choice of $\delta_i$ depending on its relative velocities with respect to its leaders, which had been inspired by \cite{CH}, was considered: \begin{equation*}\label{c}\delta_i(t)=\frac{1}{\#(\mathcal L(i))}\sum_{j\in \mathcal L(i)}\big|v_j(t)-v_i(t)\big|,\,\,\,\,\,\, i=1,2,\dots, N, \quad \mbox{and} \quad\delta_0(t)\equiv 0. \end{equation*}
Here, $\#(\mathcal L(i))$ denotes the cardinality  of the leader set $\mathcal L(i)$. 
When an agent observes  a consensus in  its leaders and itself, then it tends to give up its own preferred acceleration to follow the social leader-follower forces; otherwise, it will take an acceleration which is a   combination of  the social forces and its own preference;  the strength of its   preference is higher if it finds less consensus.  
On the other hand,   one may assume $|q_i|\leq\nu$ for some $\nu\geq0$ and all $1\leq i\leq N$. 
Under these assumptions the ratio $\frac{H}{\nu}$ expresses a   tradeoff between the social forces and individual preferences.
 Obviously,   the terms $\delta_i(t)q_i(t)$ are state-dependent  even when  $q_i(t)$'s are constant.

A simple example shows that the asymptotic flocking can fail due to a state-dependent individual preference, see \cite[Example
2.1]{L}.
Thus, a natural question is: whether it is possible to find an asymptotic  flocking in the presence of  such perturbations.
The induction method works quite well for the hierarchical C-S flocking \cite{CD1,DM,S}, but it cannot deal with the case of state-dependent
perturbations.

In order to study such a   system,   one may
  consider the ``fluctuation'' system. Let 
\begin{align}\label{reduce}\begin{array}{ll}X=
(X_1, X_2, \dots, X_N)^\top :=
(x_1-x_0, x_2-x_0, \dots, x_N-x_0)^\top,\\
V=(V_1, V_2, \dots, V_N)^\top :=(v_1-v_0, v_2-v_0, \dots,
v_N-v_0)^\top,
\end{array}\end{align}
and denote 
$Q(t)=(\delta_1(t)q_1(t), \delta_2(t)q_2(t), \dots,
\delta_N(t)q_N(t))^\top.$
Then we use the C-S model \eqref{cshlinform} and hierarchical leadership to derive  that
\begin{align*}
\begin{aligned} \label{reducedcshl}
X(t+1)&=X(t)+hV(t),  \\
 V(t+1)&=P_tV(t)+hQ(t),
\end{aligned}
\end{align*}
where $P_t:=I-h L_t$, regarded as the flocking matrix, is given by
\begin{equation*}\label{Pt}P_t=\left(\begin{array}{cccc}
\ 1-hd_1(t) & 0 & \cdots & 0\\
\ h \phi_{21}(t) & 1-hd_2(t) & \cdots & 0\\
\ \cdots & \cdots & \ddots & \cdots\\
\ h \phi_{N1}(t) & h \phi_{N2}(t) & \cdots & 1-hd_N(t)\\
\end{array}
\right).\end{equation*}
Here, the matrices $L_t$ and $P_t$  are acting on $V(t)\in {(\mathbb{R}^3)}^N$ via
the three dimensions individually.  
The crucial idea to deal with the perturbation $hQ(t)$ is a special matrix norm. For $\varepsilon\in(0,1)$, we set an $N\times N$ diagonal  matrix
 \begin{align*}\label{D}D=D_\varepsilon:=\left(\begin{array}{cccc}
\ \varepsilon^{\ell(1)}& 0 & \cdots & 0\\
\ 0 & \varepsilon^{\ell(2)} & \cdots & 0\\
\ \cdots & \cdots & \ddots & \cdots\\
\ 0 & 0 & \cdots & \varepsilon^{\ell(N)}\\
\end{array}
\right),\end{align*}
where $\ell(i)$ is the directed distance
from the leader $0$ to   $i$, i.e., the number of edges in a shortest directed path from $0$ to $i$. For any matrix $A\in \mathbb R^{N\times N}$,  we define  
\begin{equation*}\label{defineMdnorm}\|A\|_\varepsilon:=\|DAD^{-1}\|_\infty,\end{equation*} where $\|\cdot\|_\infty$ denotes the infinity norm of matrices.    For more details we refer to \cite{L}. The crucial advantage of this norm is,  for sufficient small time step $h$,
\[\|P_t\|_\varepsilon\leq 1-(1-\varepsilon)h\phi_m(t)<1,\]
 where $\phi_m{(t)}:=\min_{(i,j), j\in \mathcal L(i)}
 \phi_{ij}(x(t))>0$, see   \cite[Proposition 3.1]{L}.
This well-chosen norm enables us to   apply a boot-strapping self-bounding lemma and find   sufficient conditions  to guarantee the asymptotic flocking. For more details we refer to \cite{L}.

\subsection{Rooted leadership}


A more general framework with leadership was introduced in \cite{LX}:   the   ``rooted leadership''   which
requires that there exists a global leader which is not influenced by any
other agent but
influences them all either directly or indirectly.
\begin{definition}\cite{LX} \label{defrl}  An $(N+1)$-flock $\{0, 1, \dots, N\}$ is
said to be under {\em rooted leadership}, if there exists a root
agent, say $0$, which does not have an incoming path from others, whereas each agent in
$\{1,2,\dots,N\}$ has a directed path from $0$.
\end{definition}
   In this case, the interactions can be unidirectional or bidirectional. Obviously, the hierarchical leadership is a special case of rooted leadership.  The adjacency matrix  is in general neither symmetric nor triangular and the  induction method cannot be applied.  Use the same variable changes as in \eqref{reduce}, we find a compact form
   \begin{align}
\begin{aligned} \label{reducedcsrl}
X(t+1)&=X(t)+hV(t),  \\
 V(t+1)&=P_tV(t),
\end{aligned}
\end{align}
where $P_t:=I-h L_t$.
In \cite{LX} the   $(sp)$ matrix \cite{XG,XL} was employed  to study the flocking   of \eqref{reducedcsrl}.  As a result of the rooted leadership,  the transition matrix $P_t$   turns into an  $(sp)$ matrices when the time step $h$ is small. Based on this observation, the authors could use the infinity norm   to obtain an estimate on the iteration of transition matrices. This idea, together with the  nice self-bounding lemma, leads to the following   result.
   \begin{theorem}\cite{LX} Consider the  model \eqref{c-s-2} with interaction topology as in Definition \ref{defrl}. If $\beta<\frac{1}{2L}$, the flocking occurs for any initial data; if $\beta\geq\frac{1}{2L}$, the flocking occurs depending on the initial data.
\end{theorem}
\begin{remark}
The parameter $L$,   referred as  the ``depth'' of the graph,   is the largest distance (in the sense of graph theory)  from the leader to other agents. Note that the conditions are sufficient but not necessary; thus, the critical exponent for this case is still open.
\end{remark}

 The scenario of leadership can be observed in many physical systems, e.g.,
 flocks of flying birds and moving herds, governmental or military leadership, etc.
However,   the connectivity topology might change over time. For example, in the movement of birds flock, some individuals fly so far away from the others that they cannot see each other from time to time. In social networks, it is more realistic to assume that the neighboring agents  keep in connection only for a sequence of time slices rather than at all time instants. In \cite{LHX}, such an extended framework with joint rooted leadership was considered. \begin{definition}\cite{LHX}  The system is under  {\em joint rooted leadership}
across the time interval  $[t_1,t_2)$, $(t_1,t_2\in \mathbb N, ~t_1<t_2)$
if for the union graph of
$\{{\mathcal{G}}_{\sigma(t_1)},{\mathcal{G}}_{\sigma(t_1+1)},\dots,
{\mathcal{G}}_{\sigma(t_2-1)}\}$,  the agent $0$ does not have an incoming path from others, whereas each agent in
$\{1,2,\dots,N\}$ has a directed path from $0$.   \end{definition}
 Concerning the potential applications in engineering, this is relevant in  at least two aspects. First, the failure of connections is very common  due to the    obstacles, faults, disturbances and noise,   and the joint connectivity helps the system to endure such failures which can be recovered after a finite recovery time.  
 Second, it is  relevant to the communication costs
    because more connections entail higher costs.  In \cite{LHX}  
a flocking result was established  for joint rooted leadership, which says that the unconditional flocking occurs for $\beta<\frac{1}{2NT_0}$ where $T_0$ is  the maximum length of the time intervals for the joint connectivity.

\subsection{Alternating leaders}

 In the previous studies   the leader agent is assumed to be fixed in temporal evolution of flocks. This is not realistic.  For example, the dynamic leader-follower relation  in pigeon flocks   was discussed in \cite{NABV}.  Actually, we can often observe that the leaders  can be changed during the  migration of a migrating flock of birds. 
Of course, we can also find alternating leaders in our human social systems, for example, the periodic election of political leaders. In \cite{LH} the flocking with alternating leaders was discussed. 
Use   $\{1,2,\dots,m\}$ to label the admissible neighbor graphs with rooted leadership, then we  write the system    with a switching signal  $\sigma: \mathbb N\to \{1,2,\dots,m\} $ as follows:
\begin{align}\begin{aligned}\label{C-S1}
&x_i(t+1)=x_i(t)+hv_i(t),\qquad  i=1,2,\dots,N, \\
 &v_i(t+1)=v_i(t)+h\sum\limits_{j=1}^N \chi_{ij}^{\sigma(t)} \phi(|x_i - x_j|)\left[v_j(t)-v_i(t)\right],\\
 & \chi_{ij}^{\sigma(t)}=\left\{\begin{array}{c}     0\,,  \quad  \,\,   \, \mathrm{if}~ j\notin \mathcal L^{\sigma(t)}(i),\\ 1, \quad\,\,  \mathrm{if}~ j\in \mathcal L^{\sigma(t)}(i). \end{array}\right.  
\end{aligned}\end{align}

\begin{definition}\cite{LH}\label{defal}
The system  is under {\em rooted leadership with alternating leaders}, if
 the system  is under rooted leadership  at each time slice, but the leader agent, denoted by $r_t$, is dependent on time $t$.
\end{definition}

 Note that for the flocking with symmetric interactions or a fixed leader, such as \cite{CD1,CS1,CS2,HL,L,LHX,LX,S}, the asymptotic velocity   is a priori known, either the average of the initial velocities or just that of the leader. Thus, we can consider the dynamics of   the fluctuations around the average velocity, or around the fixed leader, to study the flocking behavior.  However, in the case of alternating leaders, one cannot a priori know the asymptotic velocity.  To overcome this difficulty, one may combine  the original system and a reference  system.
 Let \begin{align}\begin{aligned}\label{reference}  \hat x &:=(\hat x_1,  \dots, \hat x_{N-1})^\top=
(x_1-x_N,   \dots, x_{N-1}-x_N)^\top, \\
\hat v &:=(\hat v_1,  \dots, \hat v_{N-1})^\top=(v_1-v_N,   \dots,
v_{N-1}-v_N)^\top,
\end{aligned}\end{align}
 which satisfy, by \eqref{C-S1}, the reference system
\[ \label{eqhatx}
\hat x(t+1)=\hat x(t) +h\hat v(t), \quad {\hat v}(t+1)= P_{\sigma(t)}\hat v(t).
\]
On the other hand,   a compact form  of   \eqref{C-S1} reads
\begin{align}\begin{aligned}  \label{reducedflock}
x(t+1)&=x(t)+hv(t),  \\
 v(t+1)&=(I-hL_{\sigma(t)})v(t)=:F_{\sigma(t)}v(t). \end{aligned}\end{align}
In \cite{CMA1, CMA2}, the convergence estimate for the first-order consensus model was studied.  The self-bounding argument enables us to use their estimates in  \cite{CMA1, CMA2} to     find a priori estimate for $v(t)$ in \eqref{reducedflock}. This  easily turns into an estimate for $\hat v(t)$. Then the self-bounding argument can be applied on system \eqref{reference}.  The flocking result is as follows.
   \begin{theorem}\cite{LH} Consider the  model \eqref{c-s-2} with interaction topology as in Definition \ref{defal}. If $ 2\beta(N-1)^2<1$, the flocking occurs for any initial data; if $ 2\beta(N-1)^2\geq 1$, the flocking occurs depending on the initial data.
\end{theorem}

\begin{remark} In \cite{CMA1,CMA2}, the authors studied the exponential consensus  with more general interaction topologies, i.e., the rooted graph or joint rooted graph with a switch. 
 Here, a {\em rooted graph} means it has at least one spanning tree. Thus, the rooted leadership turns into a special case of rooted graph and the case of alternating leaders is a special case of rooted graphs undergoing a switch. 
   Therefore, the methodology in \cite{LH} can be easily extended to the C-S model with such a general interaction topologies, i.e.,  the rooted graph or even the joint rooted graphs with a switch. Indeed, one can easily combine the consensus estimates in \cite{CMA1,CMA2} with the argument in \cite{LH} to cover the general interaction topologies,  slightly changing   the sufficient conditions.   \end{remark}

\section{Kinetic description of Cucker-Smale model}\label{sec_kCS}
In this section, we first briefly present the derivation of the mean-field kinetic equation \eqref{KCS} from the particle system \eqref{CS} using the BBGKY hierarchy in statistical mechanics. We also discuss interactions between flocking particles and fluid.

\subsection{Derivation of the kinetic C-S model}

Let us denote $f^{N}=f^{N}(x_{1},\xi_{1},...,x_{N},\xi_{N},t)$ by the $N$-particle probability density function. Note that the density function $f^N$ is symmetric in its phase-space arguments, i.e.,
\[
f^{N}\left( \cdots,x_{i},\xi_{i},\cdots,x_{j},\xi_{j},\cdots,t\right)=f^{N}\left(\cdots,x_{j},\xi_{j},\cdots,x_{i},\xi_{i},\cdots,t\right),
\]
due to the indistinguishability of particles.

We deduce from the conservation of mass that the time evolution of $f^N$ can be written in the following form of Liouville equation:
\begin{equation}\label{eq_li}
\partial_{t}f^{N}+\sum_{i=1}^{N}\xi_{i}\cdot\nabla_{x_{i}}f^{N}+\frac{1}{N}\sum_{i=1}^{N}\nabla_{\xi_{i}}\cdot\left(\sum_{j=1}^{N}\psi\left(|x_{i}-x_{j}|\right)\,\left(\xi_{j}-\xi_{i}\right)f^{N}\right)=0.
\end{equation}
We next define the marginal distribution $f^N = f^N(x_1,\xi_1,t)$ as
\[
f^{N}\left(x_{1},\xi_{1},t\right):=\int_{\mathbb{R}^{2d(N-1)}}f^{N}\left(x_{1},\xi_{1},x_{-},\xi_{-},t\right)\,dx_{-}\,d\xi_{-},
\]
where
\[
\left(x_{-},\xi_{-}\right):=\left(x_{2},\xi_{2},...,x_{N},\xi_{N}\right).
\]
Integrating the equation \eqref{eq_li} with respect to $dx_{-}d\xi_{-}$, we find that the transport part and the forcing term of \eqref{eq_li} can be estimated as
\[
\int_{\mathbb{R}^{2d(N-1)}}\sum_{i=1}^{N}\xi_{i}\cdot\nabla_{x_{i}}f^{N}\,dx_{-}\,d\xi_{-} = \nabla_{x_{1}}\int_{\mathbb{R}^{2d(N-1)}}f^{N}\,dx_{-}\,d\xi_{-}=\xi_{1}\cdot\nabla_{x_{1}}f^{N}\left(x_{1},\xi_{1},t\right)
\]
and
\begin{align}
\begin{aligned} & \frac{1}{N}\sum_{i=1}^{N}\int_{\mathbb{R}^{2d(N-1)}}\sum_{j=1}^{N}\nabla_{\xi_{i}}\cdot\left(\psi\left(|x_{i}-x_{j}|\right)\left(\xi_{j}-\xi_{i}\right)f^{N}\right)\,dx_{-}d\xi_{-}\\
 & \qquad=\frac{1}{N}\int_{\mathbb{R}^{2d(N-1)}}\sum_{2\leq j\leq N}^{N}\nabla_{\xi_{1}}\cdot\left(\psi\left(|x_1-x_{j}|\right)\,\left(\xi_{j}-\xi_{1}\right)f^{N}\right)\,dx_{-}d\xi_{-},
\end{aligned}
\label{eq:3.2}
\end{align}
respectively. On the other hand, the symmetry property that for $j=2,3,\cdots,N$
$$\begin{aligned}
&\int_{\mathbb{R}^{2d(N-1)}}\psi\left(|x_{1}-x_{2}|\right)\left(\xi_{2}-\xi_{1}\right)f^{N}\,dx_{\text{–}}\,d\xi_{-}\cr
&\qquad =\int_{\mathbb{R}^{2d(N-1)}}\psi\left(|x_{1}-x_{3}|\right)\left(\xi_{3}-\xi_{1}\right)f^{N}\,dx_{-}\,d\xi_{-}
\end{aligned}$$
allows us to estimate \eqref{eq:3.2} as follows.
\begin{equation}
\frac{1}{N}\left(N-1\right)\int_{\mathbb{R}^{2d(N-1)}}\psi\left(|x_{1}-x_{2}|\right)\nabla_{\xi_{1}}\cdot\left(\left(\xi_{2}-\xi_{1}\right)f^{N}\right) \,dx_{-}\,d\xi_{-}.\label{eq:3.3}
\end{equation}
We now define the two-particle marginal function $g^N$ as
\[
g^{N}\left(x_{1},\xi_{1},x_{2},\xi_{2},t\right)=\int_{\mathbb{R}^{2d(N-2)}}f^{N}\text{\,}dx_{3}\,d\xi_{3}...dx_{N}\,d\xi_{N}.
\]
Then, by using the newly defined marginal function $g^N$, we rewrite \eqref{eq:3.3} as
\[
\left(1-\frac{1}{N}\right)\nabla_{\xi_{1}}\cdot\int_{\mathbb{R}^{2d}}\psi\left(|x_{1}-x_{2}|\right)\left(\xi_{2}-\xi_{1}\right) g^{N}\,dx_{2}\,d\xi_{2}.
\]
Hence we have
\[
\partial_{t}f^{N}+\xi_{1}\cdot\nabla_{x_{1}}f^{N}+\left(1-\frac{1}{N}\right)\nabla_{\xi_{1}}\cdot\int_{\mathbb{R}^{2d}}\psi\left(|x_{1}-x_{2}|\right)\left(\xi_{2}-\xi_{1}\right) g^{N}\,dx_{2}\,d\xi_{2}=0.
\]
Then by taking the mean-field limit $N \to \infty$ together with the following notations
\[
f(x_1,\xi_1,t):=\lim_{N\rightarrow\infty}f^{N}(x_{1},\xi_{1},t), \quad g(x_1,\xi_1,x_2,\xi_2,t):=\lim_{N\rightarrow\infty}g^{N}(x_{1},\xi_{1},x_{2},\xi_{2},t),
\]
we obtain that the limiting functions $f$ and $g$ satisfy
\[
\partial_{t}f+\xi_{1}\cdot\nabla_{x_{1}}f+\nabla_{\xi_{1}}\cdot\int_{\mathbb{R}^{2d}}\psi\left(|x_{1}-x_{2}|\right)\left(\xi_{2}-\xi_{1}\right)g\,dx_{2}\,d\xi_{2}=0.
\]
In order to close the above equation, we use the following assumption called {\it propagation of chaos}:
\[
g(x_{1},\xi_{1},x_{2},\xi_{2},t)=f(x_{1},\xi_{1},t)f(x_{2},\xi_{2},t).
\]
Finally, we relabel the position-velocity parameters, $(x_1,\xi_1) \mapsto (x,\xi)$ and $(x_2,\xi_2) \mapsto (y,\xi_*)$ and conclude the one-particle distribution function $f(x,\xi,t)$ satisfies the following Vlasov-type equation:
\begin{align}\label{k_CS}
\begin{aligned}
&\pa_t f + \xi \cdot \nabla_x f + \nabla_\xi \cdot (F_a(f) f) = 0, \quad (x,\xi) \in \R^d \times \R^d, \quad t >0,\cr
&F_a(f)(x,\xi,t) := K\int_{\R^d \times \R^d} \psi(|x-y|) (\xi_* - \xi)f(y,\xi_*) dy d\xi_*.
\end{aligned}
\end{align}
We can adapt the classical result of \cite{Dob} to rigorously derive the kinetic C-S equation \eqref{k_CS} due to the smoothness of the communication weight $\psi$.  More precisely, let us consider the empirical measure $\mu^N(t)$:
\bq\label{emp}
\mu^N(t) = \frac1N \sum_{i=1}^N \delta_{(x_i(t), v_i(t))},
\eq
where $(x_i(t),v_i(t))$ is a solution to the particle system \eqref{CS}. Then we can show that $\mu^N$ satisfies the equation \eqref{k_CS} in the sense of distributions, i.e., $\mu^N$ and $f$ satisfy the same equation. Before stating the mean-field limit result, we introduce several notations: Let us denote by $\mathcal{M}(\R^{2d})$ the set of positive Radon measures and fix $T > 0$. $d_{BL}(\rho_1,\rho_2)$ stands for the bounded and Lipschitz distance between two measures $\rho_1, \rho_2 \in \mathcal{M}(\R^{2d})$, i.e.,
\[
d_{BL}(\rho_1,\rho_2) := \sup_{h \in \mathcal{S}} \lt|\int_{\R^d \times \R^d} h \,d\rho_1 - \int_{\R^d \times \R^d} h \,d\rho_2\rt|,
\]
where $\mathcal{S}$ is given by
\[
\mathcal{S} := \lt\{h:\R^{2d} \to \R ~:~ \|h\|_{L^\infty} \leq 1 \quad \mbox{and} \quad \mbox{Lip(h)} := \sup_{x \neq y}\frac{|h(x) - h(y)|}{|x-y|} \leq 1  \rt\}.
\]

\begin{theorem} \cite{CCR,HL} Given $f_0 \in \mathcal{M}(\R^{2d} )$ compactly supported, take a sequence of $\mu^N_0$ of measures of the form:
\[
\mu^N_0 = \frac1N \sum_{i=1}^N \delta_{(x_i(0), v_i(0))}
\]
such that
\[
\lim_{N \to \infty} d_{BL} (\mu^N_0,f_0) = 0.
\]
Consider $\mu^N(t)$ the empirical measure \eqref{emp} with initial data $(x_i(0),v_i(0))$. Then we have
\[
\lim_{N \to \infty}d_{BL}(\mu^N(t),f(t))=0 \quad \mbox{for} \quad t \geq 0,
\]
where $f$ is the unique measure solution to the equation \eqref{k_CS} with initial data $f_0$.
\end{theorem}

\begin{remark}1. We can determine the measure solution $f$ as the push-forward of the initial density $f_0$ through the flow map generated by $(v, F_a(f))$, i.e., for any $h \in \mathcal{C}_c^1(\R^d \times \R^d)$ and $t,s \geq 0$
\[
\int_{\R^d \times \R^d}  h(x,\xi) f(x,\xi,t)\,dxd\xi = \int_{\R^d \times \R^d} h(X(0;t,x,\xi), \Xi(0;t,x,\xi)))f_0(x,\xi)\,dxd\xi,
\]
where $(X,\Xi)$ satisfy
$$\begin{aligned}
&\frac{d}{dt}X(t;s,x,\xi) = \Xi(t;s,x,\xi), \quad X(s;s,x,\xi) = x, \cr
&\frac{d}{dt}\Xi(t;s,x,\xi) = F_a(f)(X(t;s,x,\xi),\Xi(t;s,x,\xi),t), \quad \Xi(s;s,x,\xi) = \xi.\cr
\end{aligned}$$

\noindent 2. In \cite{CCHS,H}, the mean-field limits of C-S type equations with topological interactions and sharp sensitivity regions are studied. In particular, the strategy used in \cite{CCHS} can be applied to other models including nonlocal repulsive-attractive forces locally averaged over sharp vision cones. C-S model with a singular communication weight is considered in \cite{ACHL, CCH2, CCP} and the mean-field limit is studied in \cite{CCP} under suitable condition for the initial configurations. \newline

\noindent 3. Large time behavior of solutions for the equation \eqref{k_CS} is provided in \cite{CFRT, HL, HT}. Kinetic C-S type equations corresponding to \eqref{CS} and \eqref{CS-MT} with noises are treated in \cite{DFT, Choi1} showing the global existence of classical solutions near the global Maxwellian and its large-time behavior.
\end{remark}

\subsection{Interactions between flocking particles and fluids}\label{sec_kf}
In this part, we discuss the interactions between particles and its environment, i.e., fluids. Emergent phenomena of self-organized particles such as flocking, crowd, and swarming behaviors have recently received lots of attention due to the engineering, physical and biological applications. Most available literature for collective behavior deal with only the dynamics of self-organized particles as a closed system, i.e., interactions with fluids and external force fields are often ignored. However, as we can easily imagine, the dynamics of self-organized particles can be strongly influenced by neighbouring fluids and force fields, for example, water, gas and electro- magnetic waves, etc. Thus incorporating these neglected effects in the modelling of the self-organized particles will be necessary. In \cite{BCHK2}, the dynamics of flocking particles governed by the C-S model interacting with viscous compressible fluids through a drag forcing term are taken into account in the spatial periodic domain $\T^3$. More precisely, let $f = f(x,\xi,t)$ be the one-particle distribution function of the C-S flocking particles at $(x,\xi) \in \T^3 \times \R^3$ and $n = n(x,t)$, $v = v(x,t)$ be the local mass density and bulk velocity of the isentropic compressible fluid, respectively. Then the situation we mentioned above is governed by
\begin{align}\label{k_CSNS}
\begin{aligned}
&\pa_t f + \xi \cdot \nabla_x f + \nabla_\xi \cdot \lt(F_a(f) f + F_d(v) f \rt) = 0, \quad  (x,\xi) \in \T^3 \times \R^3, \quad t > 0,\cr
&\pa_t n + \nabla_x \cdot (nv) = 0,\cr
&\pa_t (n v) + \nabla_x \cdot (n v \otimes v) + \nabla_x p(n) + Lv = -\int_{\R^3} F_d(v) f\,d\xi,
\end{aligned}
\end{align}
where the pressure $p$ and the Lam\'e operator $L$ are given by
$$\begin{aligned}
&p(n) = n^\gamma \quad \mbox{with} \quad \gamma > 1,\cr
& Lv = -\mu \Delta_x v - (\mu + \lambda) \nabla_x (\nabla_x \cdot v) \quad \mbox{with} \quad \mu > 0 \quad \mbox{and} \quad \lambda + 2\mu > 0.
\end{aligned}$$
Here we assumed the coupling strength $K = 1$, and $F_a$ and $F_d$ represent the alignment and the drag forces in velocities, respectively:
$$\begin{aligned}
F_a(f)(x,\xi,t) &= \int_{\T^3 \times \R^3} \psi(|x-y|)(\xi_*-\xi)f(y,\xi_*)\,dyd\xi_* \quad \mbox{with} \quad \psi \geq 0,\cr
F_d(x,\xi,t) &= v(x,t) - \xi.
\end{aligned}$$

Recently, this kind of coupled kinetic-fluid system describing the interactions between particles and fluid has received increasing attention due to a number of their applications in the field of, for example, biotechnology, medicine, and in the study of sedimentation phenomenon, compressibility of droplets of the spray, cooling tower plumes, and diesel engines, etc \cite{BDM, RM, SG, VASG}. We refer to \cite{O,Will} for more physical backgrounds of the modelling issues in a kinetic-fluid system.

In the lemma below, we present the properties of conservation and energy estimates for the system \eqref{k_CSNS}. For details of the proof, we refer to \cite{BCHK2}.
\begin{lemma} Let $(f,n,v)$ be a classical solution to the system \eqref{k_CSNS}. Then we have
$$\begin{aligned}
&(i) \mbox{ Conservation of the mass:}\cr
&\qquad \qquad \frac{d}{dt}\int_{\T^3 \times \R^3} f\,dxd\xi = \frac{d}{dt} \int_{\T^3} n\,dx = 0.\cr
&(ii) \mbox{ Conservation of the total momentum:}\cr
& \qquad \qquad \frac{d}{dt}\lt( \int_{\T^3 \times \R^3} \xi f \,dxd\xi + \int_{\T^3} n v\,dx\rt) = 0.\cr
&(iii) \mbox{ Dissipation of the total energy:}\cr
&\qquad \qquad \frac{d}{dt}\frac12\lt( \int_{\T^3 \times \R^3}|\xi|^2 f\,dxd\xi + \int_{\T^3} n|v|^2\,dx + \frac{2}{\gamma-1}\int_{\T^3} n^\gamma \,dx \rt)\cr
&\qquad \qquad \quad = -\int_{\T^6 \times \R^6} \psi(|x-y|)|\xi - \xi_*|^2 f(x,\xi)f(y,\xi_*)\,dxdyd\xi d\xi_*\cr
&\qquad \qquad \qquad - \int_{\T^3 \times \R^3} |v - \xi|^2 f\,dxd\xi.
\end{aligned}$$
\end{lemma}

The global existence of unique strong solutions for the system \eqref{k_CSNS} is studied in \cite{BCHK2} under suitable assumptions on the initial data such as smallness and smoothness. For the large time behavior of solutions to types of equations \eqref{k_CSNS}, in \cite{Choi2}, the following Lyapunov functional $\mathcal{L}$ measuring the fluctuation of momentum and mass from the averaged quantities is introduced:
$$
\begin{aligned}
\ml(f,\rho,u) &:= \int_{\T^3 \times \R^3} |\xi - \xi_c|^2 f\,dxd\xi + \int_{\T^3} n|v - j_c|^2dx + \int_{\T^3} (n - n_c)^2dx \cr
&\quad + |\xi_c - j_c|^2,
\end{aligned}
$$
where
\[
\xi_c(t) := \frac{\int_{\T^3 \times \R^3} \xi f\,dxd\xi}{\int_{\T^3 \times \R^3} f\,dxd\xi}, \quad  j_c(t) := \frac{\int_{\T^3} n v\,dx}{\int_{\T^3} n \,dx},\quad f_c(t) := \int_{\T^3 \times \R^3} f\,dxd\xi,
\]
and
\[
n_c(t):= \int_{\T^3} n\,dx.
\]
\begin{theorem}\label{thm_kl}\cite{Choi2} Let $(f,n,v)$ be a global classical solution to the system \eqref{k_CSNS} satisfying
\begin{align*}
\begin{aligned}
&(i)\,\,\,\,\,\, \|\rho_f\|_{L^\infty(\R_+;L^{3/2}(\T^3))} < \infty \quad \mbox{where} \quad \rho_f(x,t) := \int_{\R^3} f(x,\xi,t)\,d\xi,\cr
&(ii)\,\,\,\,n(x,t) \in [0, \bar n] \quad \mbox{for all} \quad (x,t) \in \T^3 \times \R_+ \quad \mbox{and} \quad n_c(0) > 0,\cr
&(iii)\,\, v \in L^\infty(\T^3 \times \R_+) \quad \mbox{and} \quad E_0>0 \mbox{ is small enough},
\end{aligned}
\end{align*}
where $E_0$ is the initial total energy given by
\[
E_0:= \int_{\T^3 \times \R^3} |\xi|^2 f_0\,dxd\xi + \int_{\T^3} n_0 |v_0|^2 dx + \frac{2}{\gamma - 1}\int_{\T^3} n_0^\gamma \,dx.
\]
Then we have
\[
\ml(t) \leq C\ml_0e^{-\lambda t} \quad t \geq 0,
\]
where $C$ and $\lambda$ are positive constants independent of $t$.
\end{theorem}
Theorem \ref{thm_kl} shows the alignment between flocking particles and fluid velocities as time goes on exponentially fast. More precisely, it follows from conservations of masses and total momentum that
$$\begin{aligned}
&\xi_c(t) - j_c(t) \cr
&\quad = (f_c(0) + 1)\xi_c(t) - \frac{1}{n_c(0)}\lt(\int_{\T^3 \times \R^3} \xi f_0(x,\xi)\,dxd\xi - \int_{\T^3} n_0(x) v_0(x)\,dx\rt).
\end{aligned}$$
This yields
\[
\xi_c(t),\,j_c(t) \to \frac{1}{n_c(0)\lt(f_c(0) + 1\rt)}\lt(\int_{\T^3 \times \R^3} \xi f_0(x,\xi)\,dxd\xi - \int_{\T^3} n_0(x) v_0(x)\,dx\rt),
\]
as $t \to \infty$. We notice that it is natural to expect from the presence of the drag forcing term in the kinetic and fluid equations \eqref{k_CSNS}.

For the case when the fluid is incompressible, the global well-posedness and {\it a priori} estimate of large-time behaviors of solutions are studied in \cite{BCHK, BCHK3, BCHK4, CL, CK}. In particular, the density dependent drag forcing term which is more physically relevant is considered in \cite{CK2} and the global existence of strong solutions and large-time behavior are obtained. In \cite{CCK}, the dynamics of particles immersed in an incompressible fluid through local alignments is taken into account. Unlike the C-S alignment force $F_a$, each particle actively tries to align its velocity to that of its closest neighbors. For this system, the global existence of weak solutions, hydrodynamic limit corresponding to strong noise and local alignment, and large-time behavior of solutions are established. Very recently, the finite-time blow-up phenomena of classical solutions to \eqref{k_CSNS} and other related systems under suitable assumptions on the initial configurations are provided in \cite{Choi5}.

\section{Hydrodynamic descriptions for flocking behavior}
In this section, we discuss hydrodynamic models describing the behavior of flocking behavior of the C-S ensemble. We first deal with
the hydrodynamic C-S model introduced in Section 2 and then discuss its coupling with isentropic Navier-Stokes equations via the drag force.
\subsection{A hydrodynamic Cucker-Smale model}\label{sec_h_CS}
In this part, we discuss a hydrodynamic C-S model:
\begin{align}\label{h_CS}
\begin{aligned}
&\pa_t \rho + \nabla \cdot (\rho u) = 0, \quad x \in \om, \quad t > 0\cr
&\pa_t (\rho u) + \nabla \cdot (\rho u \otimes u) = \int_{\om} \psi(|x - y|) (u(y) - u(x)) \rho(x) \rho(y)\,dy,
\end{aligned}
\end{align}
subject to initial density and velocity
\[
(\rho(x, t), u(x, t))|_{t = 0} = (\rho_0(x),u_0(x)) \quad x \in \om.
\]
Here we again assumed the coupling strength $K=1$ for simplicity. Without loss of generality, we may assume that $\rho$ is a probability density function, i.e., $\|\rho(\cdot,t)\|_{L^1} = 1$ since the total mass is conserved in time. \newline

For the system \eqref{h_CS}, the global existence of classical solutions in periodic domain and moving boundary problem studied in \cite{HKK, HKK2} under suitable assumptions on the initial data and the communication weight. In one dimension, a complete description for the critical threshold to the system \eqref{h_CS} leading to a sharp dichotomy condition between global-in-time existence or finite-time blow-up of strong solutions is obtained in \cite{CCTT} which extends both, the sub- and supercritical regions derived in \cite{TT1}. Other interaction forces, such as attractive/repulsive forces in position are also considered in \cite{CCTT} for the classification of the critical thresholds to the system \eqref{h_CS}.

Inspired by \cite{CFRT}(see also Section \ref{sec_gsw}), we show the large time behavior of solutions in $L^\infty$-framework. For this, we first set spatial diameter $R^x$ and velocity diameter $R^u$ as follows.
\[
R^x(t):= \sup_{x, y \,\in \,\mbox{\small supp }\rho(\cdot,\,t)}|x- y| \quad \mbox{and} \quad R^u(t):= \sup_{x, y\, \in\, \mbox{\small supp } \rho(\cdot,\,t)}|u(x,t) - u(y,t)|.
\]
Using the above notations, we define the notion of flocking behavior for the system \eqref{h_CS}.
\begin{definition}Let $(\rho, u)$ be the solution to \eqref{h_CS}. Then the system \eqref{h_CS} exhibits global flocking if and only if the following two conditions hold.
\begin{itemize}
\item[(i)] The spatial diameter $R^x$ is uniformly bounded in time, i.e., there exists a positive constant $C$ which is independent of $t$ such that
\[
\sup_{t \geq 0} \,R^x(t) \leq C.
\]
\item[(ii)] The velocity diameter $R^u$ decays to zero as time goes to infinity:
\[
\lim_{t \to \infty} R^u(t) = 0.
\]
\end{itemize}
\end{definition}
\begin{theorem} Let $(\rho,u)$ be any smooth solutions to the system \eqref{h_CS} with compactly supported initial data $(\rho_0,u_0)$. Suppose that the initial spatial and velocity diameters satisfy
\[
R^u_0 < \int_{R^x_0}^\infty \psi(s)\,ds.
\]
Then the system \eqref{h_CS} exhibits the flocking behavior.
\end{theorem}
\begin{proof}Let us consider the following two characteristic flows:
\[
\frac{d X(t)}{dt} = u(X(t),t) \quad \mbox{and} \quad \frac{d Y(t)}{dt} = u(Y(t),t),
\]
with the initial conditions $X(0) = x$ and $Y(0) = y$ where $x, y \in \mbox{supp } \rho_0$. For notational simplicity, in the rest of estimates, we omit the time dependence of $X$ and $u$, i.e., $X := X(t)$ and $u(X) := u(X(t),t)$, similarly, it is also taken for the $Y$ and $u(Y)$. Note that
\[
\frac{d u(X)}{dt} = (\pa_t + u \cdot \nabla_x )u = \int_{\R^d} \psi(|X - y|)(u(y) - u(X))\rho(y)\,dy \quad \mbox{on supp } \rho(t).
\]

For the proof, it is enough to show that the spatial and velocity diameters satisfy the following differential inequalities:
\begin{align}\label{eq_claim}
\begin{aligned}
\frac{d}{dt}R^x(t) &\leq R^u(t),\cr
\frac{d}{dt}R^u(t) &\leq -\psi(R^x(t))R^u(t),
\end{aligned}
\end{align}
due to Theorem \ref{T1}.
First, it easily follows from the definition of the $R^x$ and $R^v$ that
\[
\frac12\frac{d}{dt}|X - Y|^2 = (X - Y) \cdot (u(X) - u(Y)) \leq R^x R^u,
\]
and this yields
\[
\frac{d}{dt} R^x(t) \leq R^u(t).
\]
Since we are dealing with the classical solutions, we can choose $X$ and $Y$ such that $R^u = |u(X) - u(Y)|$ and $R^u$ is differentiable with respect to time almost everywhere.
For the estimate of time-evolution of $R^u$, we obtain
$$\begin{aligned}
\frac12\frac{d}{dt}(R^u)^2 &= \frac12\frac{d}{dt}|u(X) - u(Y)|^2 = (u(X) - u(Y)) \cdot \lt(F(\rho)(X) - F(\rho)(Y)\rt)\cr
&=: J_1 + J_2,
\end{aligned}$$
where
\[
F(\rho)(X) := \int_{\R^d} \psi(|X - y|)(u(y) - u(X))\rho(y)\,dy. 
\]
For the estimate of $J_1$, we use the fact
\[
\lt( u(X) - u(Y)\rt) \cdot \lt( u(z) - u(X) \rt) = \lt( u(X) - u(Y)\rt) \cdot \lt( u(z) - u(Y) + u(Y) - u(X) \rt) \leq 0,
\]
for $X, Y, z \in \mbox{supp } \rho(t)$, due to the choice of $X$ and $Y$. This yields
$$\begin{aligned}
J_1 &= \int_{\R^d} \psi(|X - z|)\lt( u(X) - u(Y)\rt) \cdot \lt( u(z) - u(X) \rt)\rho(z)\,dz\cr
&\leq \psi(R^x)\int_{\R^d}\lt( u(X) - u(Y)\rt) \cdot \lt( u(z) - u(X) \rt)\rho(z)\,dz
\end{aligned}$$
Similarly, we can find
\[
J_2 \leq -\psi(R^x)\int_{\R^d}\lt( u(X) - u(Y)\rt) \cdot \lt( u(z) - u(Y) \rt)\rho(z)\,dz.
\]
Hence we have
\[
\frac12\frac{d}{dt}(R^u)^2 \leq -\psi(R^x)|u(X) - u(Y)|^2 = -\psi(R^x)(R^u)^2,
\]
where we used
\[
\int_{\R^d} \rho\,dx = 1.
\]
This completes the proof.
\end{proof}

\subsection{Hydrodynamic model for the interaction of Cucker-Smale flocking particles and fluids}

By using a similar derivation presented in Section \ref{sec_h_CS}, we can also derive the two-phase fluid model consisting of the pressureless Euler equations and the isentropic Navier-Stokes equations where the coupling is through the drag force from the coupled kinetic-fluid system \eqref{k_CSNS}. More precisely, this hydrodynamic system is governed by
\begin{align}\label{h_CSNS}
\begin{aligned}
&\pa_t \rho + \nabla_x \cdot (\rho u) = 0, \quad x \in \T^3, \quad t > 0,\cr
&\pa_t (\rho u) + \nabla_x \cdot (\rho u \otimes u) = - \rho(u-v) - \rho\int_{\T^3} \psi(|x-y|)(u(x)-u(y)) \rho(y)\,dy,\cr
&\pa_t n + \nabla_x \cdot (nv) = 0,\cr
&\pa_t (n v) + \nabla_x \cdot (n v \otimes v) + \nabla_x p(n) + Lv = \rho(u-v).
\end{aligned}
\end{align}
Here $\rho(x,t)$ and $n(x, t)$ represent the particle density and the fluid density at a domain $(x, t) \in \T^3 \times \R_+$, and $u(x, t)$ and $v(x, t)$ represent the corresponding bulk velocities for $\rho (x, t)$ and $n(x, t)$, respectively.

For the global-in-time existence of classical solutions to the system \eqref{h_CSNS}, one of main difficulties in analyzing it arises from the formation of singularities. We notice that the system $\eqref{h_CSNS}$ without the drag and nonlocal velocity-alignment forces reduces to the pressureless Euler equations, and it is well-known that the Euler equations many develop a singularity in finite time no matter how smooth the initial data are. For this reason, it is natural to extend the notion of solutions to the measure-valued solutions. Concerning this issue, an interesting question is whether the interactions with viscous fluids through the drag force can prevent the formation of the finite-time singularities, and whether the system can admit the global classical solutions.

In \cite{HKK}, the global existence of classical solutions for the pressureless Euler/incompressible Navier-Stokes equations with the nonlocal alignment forces and its large-time behavior are studied. It is interesting that the ({\it a priori}) estimate of time behavior of solutions plays an important role in constructing the global-in-time solutions. For the system \eqref{h_CSNS} without the alignment force, i.e., $\psi \equiv 0$, the global existence and uniqueness of classical solutions and {\it a priori} estimate of large-time behavior of solutions showing that the two fluid velocities are aligned exponentially fast  are obtained in \cite{CK3}. The strategy used in \cite{CK3} can be directly applied to the system \eqref{h_CSNS}, and in particular, we can deduce from \cite{CK3} the following {\it a priori} estimate for the large-time behavior of solutions to the system \eqref{h_CSNS}.

\begin{theorem}Let $(\rho,u,n,v)$ be the classical solutions to the system \eqref{h_CSNS} satisfying
\begin{align}\label{condi_h_CSNS}
\begin{aligned}
&(i)\,\,\,\,\,\, \rho, \,n, \,v \in L^\infty(\T^3 \times \R_+).\cr
&(ii) \,\,\,\, \rho_c(0), \,n_c(0) \in (0,\infty)\quad  \mbox{and} \quad \widetilde{E_0} > 0 \mbox{ is small enough},
\end{aligned}
\end{align}
where $\widetilde{E_0}$ is an initial total energy given by
\[
\widetilde{E_0}:= \int_{\T^3} \rho_0|u_0|^2\,dx + \int_{\T^3} n_0 |v_0|^2 dx + \frac{2}{\gamma - 1}\int_{\T^3} n_0^\gamma \,dx.
\]

Then we have
\[
\widetilde{\mathcal{L}}(t) \leq C\widetilde{\mathcal{L}}_0e^{-\lambda t}, \quad t \in [0,T],
\]
for some constants $C$ and $\lambda > 0$, where
\[
\widetilde{\mathcal{L}}(t)= \int_{\T^3} \rho |u-m_c|^2 dx + \int_{\T^3} n| v - j_c|^2 dx + |m_c - j_c|^2 + \int_{\T^3} (n-n_c)^2 dx,
\]
and
\[
m_c(t) := \frac{\int_{\T^3} \rho u\,dx}{\int_{\T^3} \rho\,dx}.
\]
\end{theorem}
It is worth noticing that we do not require that the $L^\infty(\T^3)$-norms of solutions $\rho$, $n$, and $v$ should be small, we need only the small initial total energy.

As mentioned in Remark \ref{rem_max}, we can derive the isothermal Euler equations coupled with Navier-Stokes equations from the kientic-fluid system \eqref{k_CSNS} by considering the strong noise and the local alignment instead of the velocity-alignment force $F_a(f)$. We notice that for the hydrodynamic limit to be rigorously derived (and not only formally) within the framework of relative entropy techniques, one of the main challenges is to establish the global existence of strong solutions of the fluid equation. To be more precise, the standard argument for the hydrodynamic limit is based on the weak-strong stability employing a relative entropy functional and holds as long as there exist global weak solutions to the kinetic-fluid equations and strong solutions to the fluid-fluid equations. However, as we briefly mentioned as before, solutions of Euler-type equations are well-known to possibly develop a singularity in a finite-time no matter how smooth the initial data are.

For the isothermal Euler/incompressible Navier-Stokes equations, the global existence and uniqueness of classical solutions are studied in \cite{Choi3} by reinterpreting the drag forcing term as the relative damping and extracting the smoothing effect of viscosity in the Navier-Stokes equations. This yields that its rigorous derivation from Vlasov-Fokker-Planck/ incompressible Navier-Stokes equations with local alignment forces for some particular regime of the dispersed phase obtained in \cite{CCK} holds for all time. For the interactions with compressible fluids, i.e., isothermal Euler/compressible Navier-Stokes equations, the global-in-time existence of classical solutions and its large-time behavior are obtained in \cite{Choi4}.

%
%
%
%
\section*{Acknowledgement}
The work of S.-Y. Ha is supported by the Samsung Science and Technology Foundation under Project Number SSTF-BA1401-03. The work of Y.-P. Choi is supported by Engineering and Physical Sciences Research Council(EP/K008404/1) and ERC-Starting Grant HDSPCONTR ``High-Dimensional Sparse Optimal Control". The work of  Z. Li was supported by the National Natural Science Foundation of China grant 11401135.

%
%
%

\end{document}